\documentclass[a4paper,11pt]{amsart}

\usepackage{graphicx}
\usepackage{amsmath}
\usepackage{amssymb}
\usepackage{txfonts}
\usepackage{hyperref}

\newtheorem{thm}{Theorem}
\newtheorem{lem}{Lemma}%
\newtheorem{prop}{Proposition}%
\theoremstyle{definition}

\theoremstyle{remark}
\newtheorem{remark}{Remark}[section] %

\theoremstyle{plain}

\numberwithin{equation}{section}

\def\CC{{\mathbb C}}

\def\HH{{\mathbb H}}

\def\NN{{\mathbb N}}
\def\QQ{{\mathbb Q}}

\def\RR{{\mathbb R}}
\def\TT{{\mathbb T}}

\def\ZZ{{\mathbb Z}}

\def\scrD{{\mathcal D}}

\def\scrM{{\mathcal M}}

\def\Re{\operatorname{Re}}
\def\Im{\operatorname{Im}}

\def\e{\mathrm{e}}
\def\i{\mathrm{i}}

\def\C{\operatorname{C{}}}
\def\G{\operatorname{G{}}}

\def\SL{\operatorname{SL}}

\def\sgn{\operatorname{sgn}}

\def\SLZ{\SL(2,\ZZ)}
\def\SLR{\SL(2,\RR)}

\def\USLR{\widetilde\SL(2,\RR)}

\def\Del{\Delta_1(4)}

\def\DelSLR{\Del\backslash\USLR}

\title{Short incomplete Gauss sums and rational points on metaplectic horocycles}
\author{Emek Demirci Akarsu}
\address{School of Mathematics, University of Bristol,
Bristol BS8 1TW, U.K.\newline
\rule[0ex]{0ex}{0ex} \hspace{8pt}{\tt e.demirciakarsu@bristol.ac.uk}}
\date{\today}
\thanks{E.D.A.\ is supported by a Turkish Ministry of Education doctoral training grant.}

\begin{document}

\begin{abstract}In the present paper we investigate the limiting behavior of short incomplete Gauss sums at random argument as the number of terms goes to infinity. We prove that the limit distribution is given by the distribution of theta sums and differs from the limit law for long Gauss sums studied by the author and Marklof. The key ingredient in the proof is an equidistribution theorem for rational points on horocycles in the metaplectic cover of $\SLR$.
\end{abstract}

\maketitle

\section{Introduction \label{secIntro}}
The sum we are concerned with is the incomplete Gauss sum
\begin{equation}\label{sgauss}
g_f(p, q, N)=\sum_{h\in\ZZ} f\bigg(\frac{h}{N}\bigg)\, e_q(p h^2),
\end{equation}
where $q\in\ZZ$, $p\in\ZZ_q^\times=\{p\in\ZZ_q :\,\gcd(p,q)=1\}$, $\ZZ_q=\ZZ/q\ZZ$ and $e_q(t)=\e^{\frac{2\pi \i t}{q}}$. The function $f:\RR\to\RR$ is assumed to be Riemann integrable with compact support. The classical example of an incomplete Gauss sum is when $f$ is the characteristic function of a subinterval of $[0,1]$, see for instance \cite{Lehmer76, Fiedler77, Oskolkov91, Evans03, Paris05, Paris08}. The exponential sum here is quadratic in $h$, see \cite{Montgomery95} for the more difficult case of higher powers.

The incomplete Gauss sum is a special case of a theta sum
\begin{equation}\label{thetasum} S_f(x,N)=\sum_{h\in\ZZ}f\bigg( \frac{h}{N}\bigg)\,e(x h^2),\end{equation}
evaluated at rational point $x=\frac{p}{q}$, where $e(t)=\e^{2 \pi \i t}$. If we take $f$ to be the characteristic function of $(0,1]$, we have the classical case
\begin{equation}S(x,N)=\sum_{h=1}^{N} e(x h^2).
\end{equation}
The limiting distribution of theta sum $|S(x,N)|$ has been studied by Jurkat and van Horne \cite{Jurkat81,Jurkat82,Jurkat83}, for $x$ uniformly distributed on $\TT=\RR\backslash\ZZ$, and as $N\to\infty$ the full distribution of $S_f(x,N)$ in the complex plane has been investigated by Marklof \cite{Marklof99}. Cellarosi \cite{Cellarosi11} has recently shown that $S(x,\left[tN\right])$, $t\in[0,1]$ converges to a random process. 

The limit law of the value distribution of the sum $g_f(p,q,N)$ when $N$ is of the same order of magnitude as $q$ has been studied by the author and Marklof \cite{DemirciMarklof}. The limit distribution in this case is given by the distribution of a family of periodic functions. Our aim in this paper is to investigate the limiting behavior of short incomplete Gauss sums for $N=o(q)$.

Our main result is the following theorem.
\begin{thm}\label{thm1} Fix $a\in\NN$ and a subset $\scrD \subset \TT$ with boundary of measure zero. For each $q\in\NN$ choose $p\in\ZZ_q^\times\cap q \scrD$ at random with uniform probability. Given $f:\RR\to\RR$ Riemann integrable and compactly supported, there exists a probability measure $\nu_f$ on $\CC$ such that as $q\to\infty$ along any subsequence with $q=aq'$, $q'$ prime, $\frac{N}{q}\to 0$ and $\frac{N^{4/3}}{q}\to\infty$, for any bounded continuous function $F : \CC\to\RR$, we have 
$$\frac{1}{|\ZZ_q^\times\cap q \scrD|}\sum_{p\in\ZZ_q^\times\cap q\scrD} F\bigg(\frac{{g_f}(p, q, N)}{\sqrt{N}}\bigg) \rightarrow \int_{\CC} F(z)\,\nu_f(d z).$$
\end{thm}

The limiting probability measure $\nu_f$ here is the same as for the theta sum $S_f(x,N)$ in \cite{Jurkat81,Jurkat82, Jurkat83, Marklof99} and different from the range $N\asymp q$, discussed in \cite{DemirciMarklof}. In particular we have the tail $\int_{|z|>R}\nu_f(d z)\sim c_f\, R^{-4}$ as shown by Jurkat and van Horne \cite{Jurkat81, Jurkat82, Jurkat83}, and Marklof \cite{Marklof99}. We believe that the assumption $\frac{N^{4/3}}{q}\to\infty$ can be removed as long as $N\to\infty$, cf. Remark \ref{remark1} below and Figure \ref{figboundry}.

A more general version of Theorem \ref{thm1} is stated in Section \ref{maintheorem}.

The proof of the theorem for smooth functions comes from equidistribution of rational points on metaplectic horocyles (Section \ref{rational}) and the relation between theta series and short incomplete Gauss sums explained in Section \ref{maintheorem}. The proof is provided in Section \ref{smooth} for smooth weights. For Riemann integrable weights the proof uses mean-square estimates, see Section \ref{riemann}.

To illustrate Theorem \ref{thm1}, we have numerically computed the value distribution of real and imaginary part of the short incomplete Gauss sum, see Figures \ref{fig6007}-\ref{figN=q}. The difference between the limit distribution of incomplete Gauss sum (the case $N\asymp q$) and short case (the case $N=o(q)$) is that in the former the limiting distribution has compact support \cite{DemirciMarklof}, cf. Figure \ref{figN=q}, while in the later the limit distribution has long-tailed distribution, cf. Figure \ref{fig6007} and Figure \ref{figboundry}. More details in Section \ref{numerics}.

\begin{figure}
\begin{center}
\includegraphics[width=0.49\textwidth]{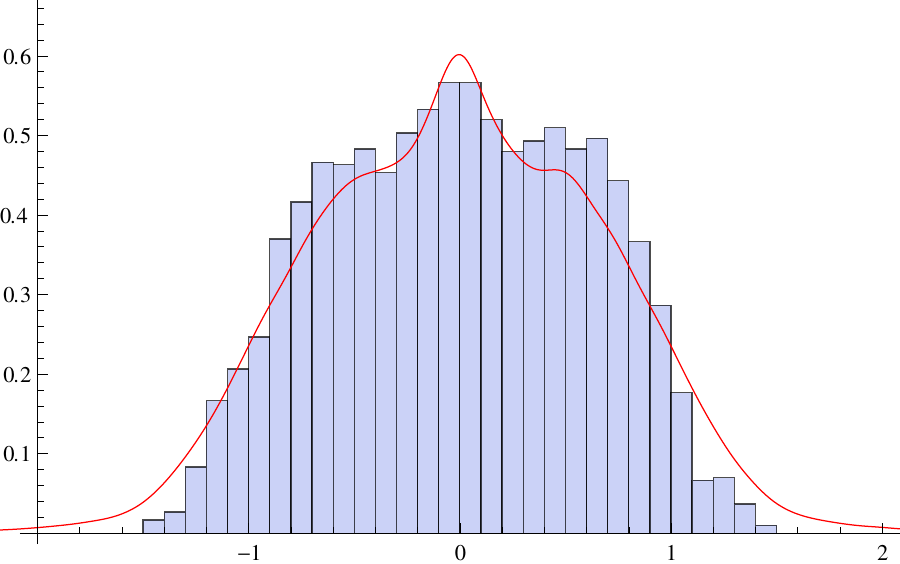}
\includegraphics[width=0.49\textwidth]{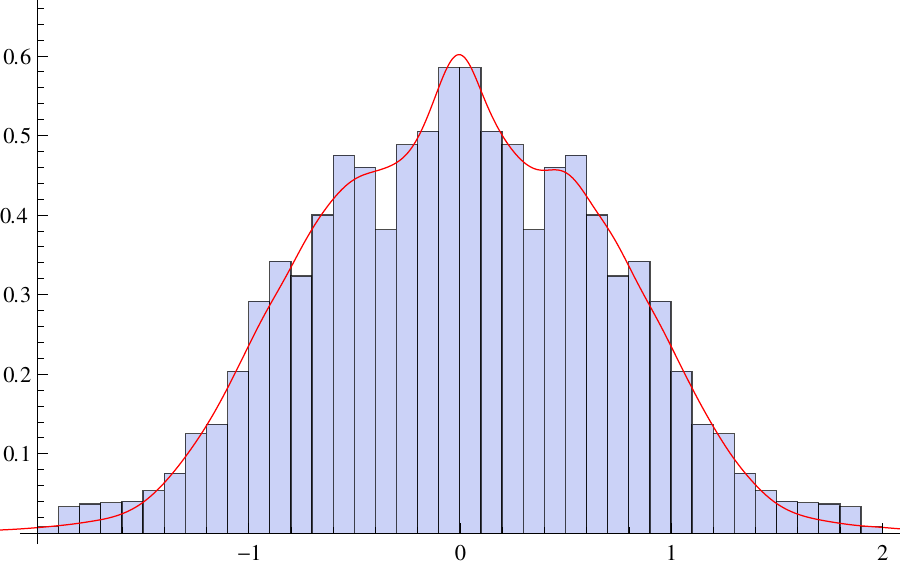}
\end{center}
\caption{The histogram on the left shows the value distribution of the real part of $\frac{g_f(p,q,N)}{\sqrt N}$ where $f$ is the characteristic function of the unit interval and $N=\left\lfloor \frac{1}{\sqrt{7}}\,q^{7/8}\right\rfloor$, $q=6007$ (prime) and $p$ uniformly distributed in $\ZZ_q^\times$, where $\varphi(q)=6006$. The histogram on the right shows the imaginary part of the corresponding quantities. The continuous curves represent a numerical approximation to the real part and imaginary part of the probability density of the random variable $S_f(x,N)$ with $x$ uniformly distributed in [0,1].}\label{fig6007}
\end{figure}
\begin{figure}
\begin{center}
\includegraphics[width=0.49\textwidth]{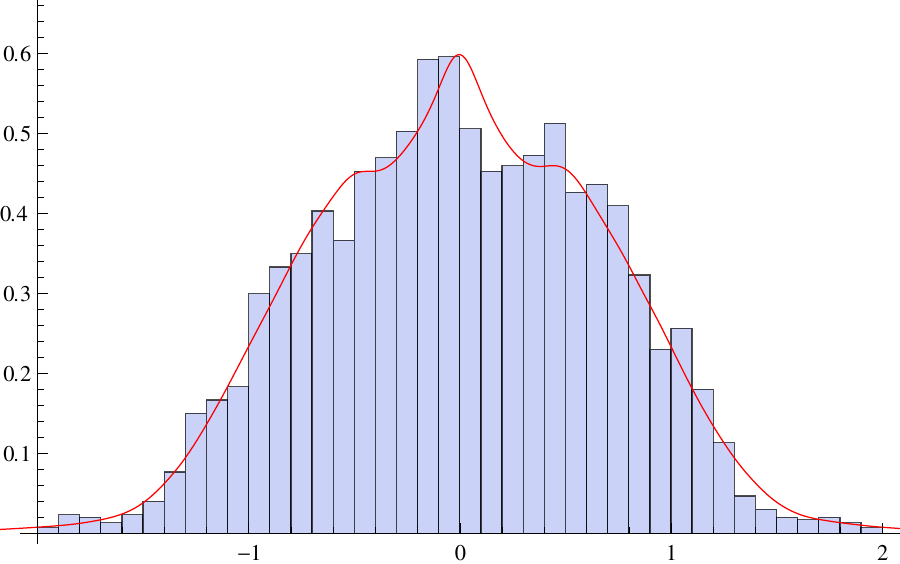}
\includegraphics[width=0.49\textwidth]{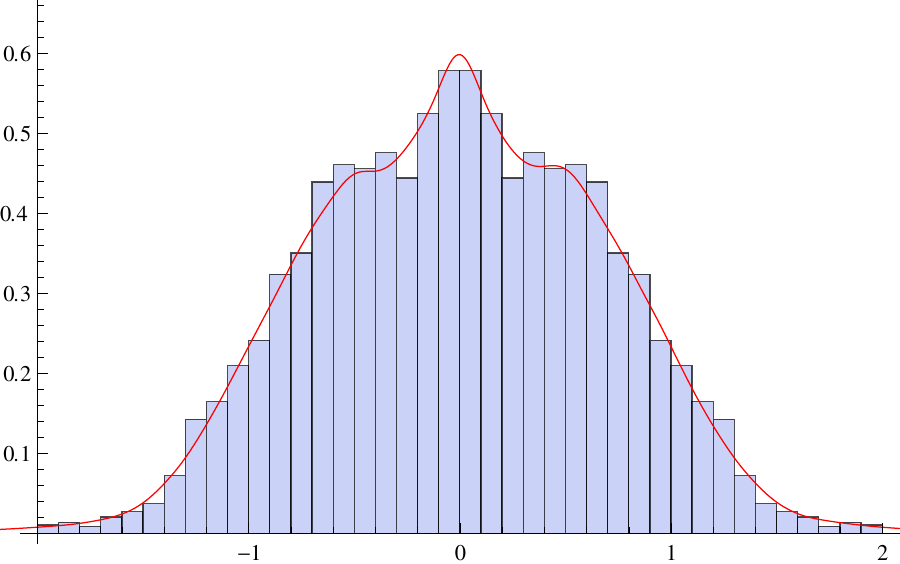}
\end{center}
\caption{The histograms and the curves represent the same as Figure \ref{fig6007} but now for $N=\left\lfloor \frac{1}{\sqrt{7}}\,q^{3/4}\right\rfloor$, which is outside the range of validity of our theorem.}\label{figboundry}
\end{figure}
\begin{figure}
\begin{center}
\includegraphics[width=0.49\textwidth]{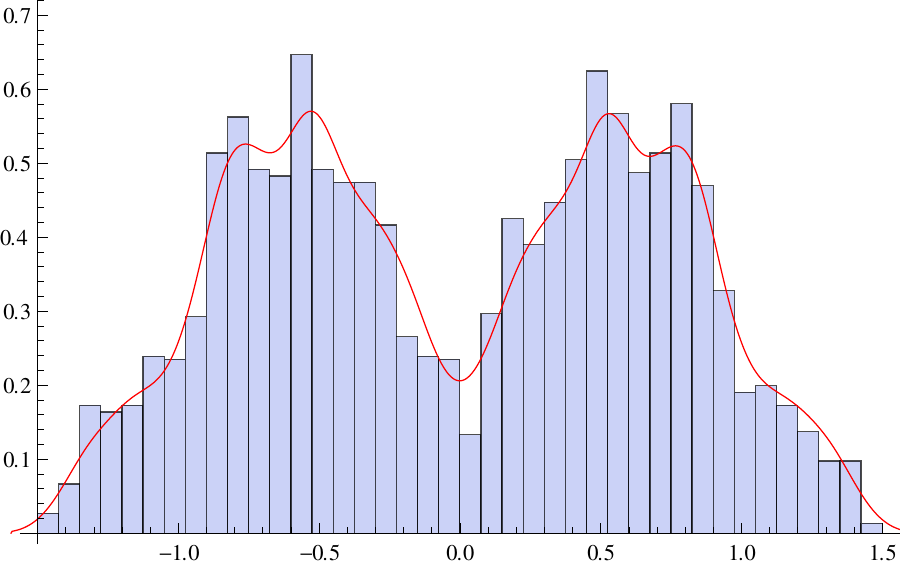}
\includegraphics[width=0.49\textwidth]{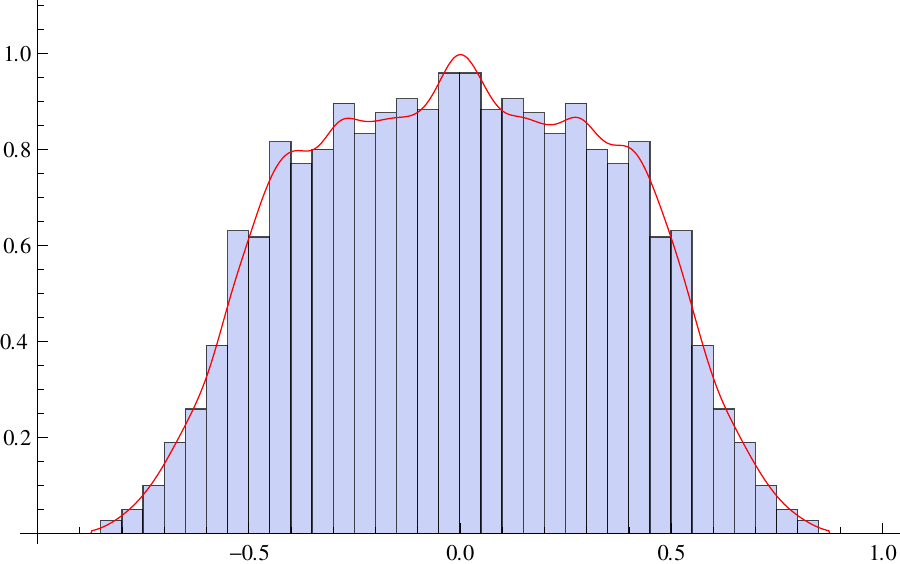}
\end{center}
\caption{The histograms show the value distributions of the real part and imaginary part of $\frac{g_f(p,q,N)}{\sqrt N}$ where $f$ is the characteristic function of the unit interval and $N=\left\lfloor \frac{1}{\sqrt{7}}\,q\right\rfloor$, we choose $q=6029$ (prime), where $\varphi(q)=6028$. The continuous curves represent a numerical approximation to the real part and imaginary part of the probability density of the random variable $7^{1/4}\,Y\,G_\varphi(x)$ with $x$ uniformly distributed in [-1/2,1/2]. The notations $Y$ and $G_\varphi$ are explained in Section \ref{numerics}.}\label{figN=q}
\end{figure}

As mentioned earlier, the key step in the proof of Theorem \ref{thm1} is the equidistribution of rational points on horocyles as the period of horocycles becomes large. The precise result is stated as Theorem \ref{mtplctcEq} in Section \ref{rational}. For the purpose of this introduction, we will for now restrict our attention to horocycles on the homogeneous space $\Gamma\backslash G$ with $ G:=\SLR $ and $\Gamma:=\SLZ$. The Iwasawa decomposition gives a convenient parametrization of $\SL(2,\RR)$:
For given $g\in \SL(2,\RR)$ there is a unique $x,y,\phi$ as below

\begin{equation}g=
\left(\begin{array}{cc}
1&x\\
0&1\\	
\end{array}\right)
\left(\begin{array}{cc}
y^{1/2}&0\\
0&y^{-1/2}\\	
\end{array}\right)
\left(\begin{array}{cc}
\cos(\phi)&\sin(\phi)\\
-\sin(\phi)&\cos(\phi)\\	
\end{array}\right)
\end{equation}
where $z=x+\i y\in \HH:=\{z\in\CC: \Im z>0\}$ and $\phi\in[0,2\pi)$.

The order of $\ZZ_q^\times$ is denoted by Euler's totient function $\varphi(q)$.

\begin{thm}\label{modgeqthm}Let $f\in C_0^{\infty}(\TT\times\Gamma\backslash G)$ (infinitely differentiable and of compact support). Then the following holds as $y\to 0$, $q\to \infty$, $q^2y\rightarrow\infty$ with $q$ prime:
 \begin{equation}\frac{1}{\varphi(q)}\sum_{p\in \ZZ_q^\times}f\bigg(\frac{p}{q},\frac{p}{q}+\i y,0 \bigg)=\int_{\TT\times\Gamma\backslash G} f(\xi,z,\phi)\,d\xi\,d\mu(z,\phi)+O(q^{3/2} y)+o\bigg(\frac{1}{\sqrt{q^2y}}\bigg)
.\end{equation}
where $d\mu=dx\,dy\,d\phi/{y^2}$.
\end{thm}

\begin{remark}\label{remark1}The error term $O(q^{3/2}y)$ can most likely be removed, e.g. by the methods employed in \cite{Marklof03}. Note also that for $q$ prime, $\frac{p}{q}+\i y$ with $y=\frac{1}{q}$ are Hecke points, which are well known to be uniformly distributed on $\Gamma\backslash\G$ \cite{Clozel, Eskin, Goldstein}.

\end{remark}

We extend the result of Theorem \ref {modgeqthm} to the distribution of rationals on metaplectic horocyles, compare Theorem \ref{mtplctcEq}. The proof uses the Weil bounds on Kloosterman and Sali\'e sums, and exploits Marklof's equidistribution theorem \cite{Marklof} for metaplectic horocycles.

\section {Preliminaries}\label{preliminaries}

The definitions and related properties below come from \cite{Marklof}.

We will define the action of $\SLR$ on the upper half-plane $\HH$ by fractional linear transformations,
 i.e.,
\begin{equation} \label{sltwo}
g: z \mapsto gz = \frac{az+b}{cz+d},\quad 
g=\left(\begin{array}{cc}
a & b \\ c & d
\end{array}\right)
\in \SL(2,\RR).
\end{equation}
The universal covering group of $\SL(2,\RR)$ is denoted by $\widetilde{\SL}(2,\RR)$. That geometrically is $\widetilde{\SL}(2,\RR)\simeq\HH\times\RR$ and can be identified as
\begin{equation}\label{betafunct}
\widetilde{\SL}(2,\RR)=\left\{[g,\beta_g]:g\in\SL(2,R),\text{$\beta_g$ a continuous function on $\HH$ s.t.    $\e^{\i\beta_g(z)}=\frac{cz+d}{|cz+d|}$}\right\}
\end{equation}
with the property that
\begin{equation}[g,\beta_g](z,\phi)=(gz,\phi+\beta_g(z)).
\end{equation} 
The product of two elements is given by
\begin{equation}
[g, \beta_g][g',\beta'_{g'}]=[gg',\beta''_{gg'}],\quad\beta''_{gg'}(z):=\beta_g(g'z)+\beta'_{g'}(z).
\end{equation} 
The inverse of $[g,\beta_g]$ is $[g^{-1},\beta'_{g^{-1}}]$ with the function $\beta'_{g{-1}}(z)=-\beta_g(g^{-1}z)$. We may identify $\widetilde{\SL}(2,\RR)$ with $\HH\times\RR$ via $[g,\beta_g]\longmapsto(z,\phi)=(g\i,\beta_g(\i))$.

To define the relevant discrete subgroups in $\widetilde{\SL}(2,\RR)$, consider the following congruence subgroups of $\SL(2,\ZZ)$
\begin{equation}\Gamma_0(4)=\left\{\left(\begin{array}{cc}
a&b\\
c&d	
\end{array}\right)\in\SL(2,\ZZ):c\equiv 0\bmod 4\right\}
\end{equation} and
\begin{equation}\Gamma_1(4)=\left\{\left(\begin{array}{cc}
a&b\\
c&d	
\end{array}\right)\in\Gamma_0(4):a\equiv d\equiv 1\bmod4\right\}.
\end{equation}
A discrete subgroup of $\widetilde{\SL}(2,\RR)$ is given by
\begin{equation}\Delta_1(4)=\left\{[g,\beta_g]:g\in\Gamma_1(4),\text{$\beta_g$ a continuous function on $\HH$ s.t. $\e^{\i\beta_g(z)}=\left(\frac{c}{d}\right)\epsilon_g(z)^{1/2}$}\right\}
\end{equation}
with $\left(\frac{c}{d}\right)\epsilon_g(z)^{1/2}=\left(\frac{c}{d}\right)\epsilon_g(e)^{\i(arg(cz+d))/2}$, where $\epsilon_a$ takes the value 1 or $\i$ when $a\equiv 1 \bmod 4$ or $a\equiv 3 \bmod 4$, respectively, and $\left(\frac{c}{d}\right)$ denotes the generalized quadratic residue symbol. For odd prime $p$ and an integer $a$ prime to $p$, the quadratic residue symbol $\left(\frac{a}{p}\right)$ has the value +1, if $a$ is a square modulo  $p$, and -1, if this is not the case. For an  integer $a$ and an odd integer $b$ the generalized quadratic residue symbol $\left(\frac{a}{b}\right)$ is characterized as Jacobi symbol defined as follows. Let $b$ be a positive odd integer with prime factorization $\prod^{r}_{i=1}b_i$. For $a\in\ZZ$, we define $\left(\frac{a}{1}\right)=1$ and 
\[ \left(\frac{a}{b}\right)=\left(\frac{a}{b_1}\right)...\left(\frac{a}{b_r}\right)\]
a product of Legendre symbol.

The generalized quadratic residue symbol (or Jacobi symbol) $\left(\frac{a}{b}\right)$ is characterized by the following properties (cf.~\cite{Shimura73}):
\begin{itemize}
\item $\left(\frac{a}{b}\right)=0$ if $\gcd(a,b)\neq 1$,
\item If $b$ is an odd prime, $\left(\frac{a}{b}\right)$ coincides
with the ordinary quadratic residue symbol,
\item If $b>0$, $\left(\frac{\cdot}{b}\right)$ defines a character
modulo $b$,
\item If $a\neq 0$, $\left(\frac{a}{\cdot}\right)$ defines a character
modulo a divisor of $4a$, whose conductor is the conductor of
$\QQ(\sqrt a)$ over $\QQ$, 
\item $\left(\frac{a}{-1}\right)=\sgn a$,
\item $\left(\frac{0}{\pm 1}\right)= 1$.
\end{itemize}
In particular $\left(\frac{a}{b}\right)^2$=1, if $\gcd(a,b)=1$.
The quotient
\begin{equation}\scrM=\DelSLR=\{\Del\;\widetilde g\;:\; \widetilde g\in \USLR \}
\end{equation}
is of finite measure with respect to the invariant measure $\frac{dx\,dy\,d\phi}{y^2}$ and $\mu(\scrM)=8\pi^2$.

\section{Rational points on horocycles in the metaplectic group}\label{rational} 

Our main result in this section, which is the extended version of Theorem \ref{modgeqthm}, is the next theorem. We need the following definitions appearing in the theorem.

The $\nu$-th Hermite function $h_\nu$ reads
\begin{equation}\label{hermite}
h_\nu(t)= (2^{\nu-1} \nu!)^{-1/2} H_\nu(2\pi^{1/2}t)\,\e^{-2\pi t^2},
\end{equation}
with the Hermite polynomial 
$$
H_\nu(t)=(-1)^\nu \e^{t^2} \frac{d^\nu}{dt^\nu} \e^{-t^2}.
$$
Let us consider the theta function
\begin{equation}
\theta_\nu(z)=y^{1/4}\sum_{n\in \ZZ} h_\nu(n y^{1/2})\, e(n^2 x),
\end{equation} 
with the transformation property (c.f. \cite{Shimura93})
\begin{equation} \label{trasprop}
\theta_\nu(gz)= j_g(z)^{2\nu+1} \theta_\nu(z),\quad
\text{for every $g\in\Gamma_0(4),$}
\end{equation}
where $j_g(z)=\epsilon_d^{-1} \left(\frac{c}{d}\right)
   \left(\frac{cz+d}{|cz+d|}\right)^{1/2}$. The function
$\theta_\nu(z,\phi)=\theta_\nu(z)\,\e^{-\i (2\nu+1)\phi/2}$
is invariant under $\Delta_1(4)$, that is we define $\theta_\nu(z,\phi)$ as a function on 
$\USLR$.

\begin{thm}\label{mtplctcEq} Fix $a\in\NN$. Let $f\in\C_0^\infty\big(\TT\times\scrM\big)$, $p,q\in\NN$, $\gcd(p,q)=1$, $y\in\RR_{>0}$ and $\widetilde\phi\in\pi(\ZZ+\frac12)$. Then the following holds as $y\to 0$, $q \to \infty$ along any subsequence with $q=aq'$, $q'$ prime :
\begin{enumerate}
\item[(i)] If $q^2y\to \infty$ and $q\equiv 0\bmod 4$, then for every $\widetilde\phi\in\{\pm\frac{\pi}{2},\pm\frac{3 \pi}{2}\}$
\begin{equation}\begin{split}
\frac{1}{\varphi(q)} \sum_{\substack{p\in\ZZ_q^\times\\ \phi_{p,q}=\widetilde\phi}} f\bigg(\frac{p}{q},\bigg(\frac{p}{q}+\i y,0\bigg)\bigg)
&=\frac{1}{32\pi^2}
\int_{\TT\times{\scrM}} f(\xi,z,\widetilde\phi)\;d\xi\;d\mu(z,\widetilde\phi)\\
&+
\frac{1}{32 \pi^2}
\bigg[\sum_{\nu=0}^\infty h_{2\nu}(0) \int_{\TT\times{\scrM}} f(\xi,z,\widetilde\phi) 
\Re\{ \theta_{2\nu}(z,\widetilde\phi)\}\,
d\xi\; d\mu(z,\widetilde\phi)\bigg]\; \bigg(\frac{1}{q^2y}\bigg)^{1/4} \\
&+O(q^{3/2}y)+o\bigg(\frac{1}{\sqrt{q^2y}}\bigg).
\end{split}\end{equation}

\item[(ii)] If $q^2y\to \infty$ and $q\equiv 1\bmod 2$, then for every $\widetilde\phi\in\{\pm\frac{\pi}{2}\}$
\begin{equation}\begin{split}
\frac{1}{\varphi(q)} \sum_{\substack{p\in\ZZ_q^\times\\ \phi_{p,q}=\widetilde\phi}} f\bigg(\frac{p}{q},\bigg(\frac{p}{q}+\i y,0\bigg)\bigg)
&=\frac{1}{16\pi^2}
\int_{\TT\times{\scrM}} f(\xi,z,\widetilde\phi)\;d\xi\;d\mu(z,\widetilde\phi)\\
&+
\frac{1}{16 \pi^2}
\bigg[\sum_{\nu=0}^\infty h_{2\nu}(0) \int_{\TT\times{\scrM}} f(\xi,z,\widetilde\phi) 
\Re\{ \theta_{2\nu}(z,\phi)\}\,
d\xi\; d\mu(z,\widetilde\phi)\bigg]\; \bigg(\frac{1}{4q^2y}\bigg)^{1/4}\\
&+O(q^{3/2}y)+o\bigg(\frac{1}{\sqrt{q^2y}}\bigg).
\end{split}
\end{equation}
\item[(iii)] If $q^2y\to \infty$ and $q\equiv 2\bmod 4$, then for every $\widetilde\phi\in\{\pm\frac{\pi}{2}\}$
\begin{equation}\begin{split}
\frac{1}{\varphi(q)} \sum_{\substack{p\in\ZZ_q^\times\\ \phi_{p,q}=\widetilde\phi}} f\bigg(\frac{p}{q},\bigg(\frac{p}{q}+\i y,0\bigg)\bigg)
&=\frac{1}{16\pi^2}
\int_{\TT\times{\scrM}} f(\xi,z,\widetilde\phi)\;d\xi\;d\mu(z,\widetilde\phi)\\
&+
\frac{1}{16 \pi^2}
\bigg[\sum_{\nu=0}^\infty h_{2\nu}(0) \int_{\TT\times{\scrM}} f(\xi,z,\widetilde\phi) 
\Re\{ \theta_{2\nu}(z,\widetilde\phi)\}\,
d\xi\; d\mu(z,\widetilde\phi)\bigg]\; \bigg(\frac{1}{q^2y}\bigg)^{1/4}\\
&+O(q^{3/2}y)+o\bigg(\frac{1}{\sqrt{q^2y}}\bigg).
\end{split}
\end{equation}
\end{enumerate}
\end{thm}

The proof of this theorem is presented later on in this section. We first state some auxiliary results.

Let us now define $\widetilde\gamma_0$ and $\widetilde n_-(x)$ as follows:
\begin{equation}
\widetilde\gamma_0 =\bigg[ \begin{pmatrix} 0 & -\frac12 \\ 2 & 0 \end{pmatrix}, \arg \bigg],
\end{equation}
\begin{equation}
\widetilde n_-(x) =\bigg[ \begin{pmatrix} 1 & x \\ 0 & 1 \end{pmatrix}, 0 \bigg].
\end{equation}

\begin{prop}\label{lemEqui} Fix $a\in\NN$. Let $f\in\C_0^\infty\big(\TT\times\scrM\big)$ bounded, $p,q\in\NN$, $\gcd(p,q)=1$, $y\in\RR_{>0}$ and $\widetilde\phi\in\pi(\ZZ+\frac12)$. Then the following holds as $y\to 0$, $q \to \infty$ along any subsequence with $q=aq'$, $q'$ prime :

\begin{enumerate}
\item[(i)] If $q^2y\to \infty$ and $q\equiv 0\bmod 4$, then for every $\widetilde\phi\in\{\pm\frac{\pi}{2},\pm\frac{3 \pi}{2}\}$
\begin{equation}\label{conv1}
\frac{1}{\varphi(q)} \sum_{\substack{p\in\ZZ_q^\times \\ \phi_{p,q}=\widetilde\phi}} f\bigg(\frac{p}{q},\bigg(\frac{p}{q}+\i y,0\bigg)\bigg)
= \frac{1}{4}\int_{\TT^2} f\bigg(\xi, \bigg(x+\i \frac{1}{q^2y}, \widetilde\phi \bigg)\bigg)\, d\xi\,dx+O(q^{3/2}y) .
\end{equation}
\item[(ii)] If $q^2y\to \infty$ and $q\equiv 1\bmod 2$, then $\widetilde\phi\in\{\pm\frac{\pi}{2}\}$
\begin{equation}\label{conv2}
\frac{1}{\varphi(q)} \sum_{\substack{p\in\ZZ_q^\times \\ \phi_{p,q}=\widetilde\phi}} f\bigg(\frac{p}{q},\bigg(\frac{p}{q}+\i y,0\bigg)\bigg)
=\frac{1}{2}\int_{\TT^2} f\bigg(\xi, \widetilde\gamma_0\bigg(x+\i \frac{1}{4q^2y}, \widetilde\phi\bigg)\bigg)\, d\xi\,dx+O(q^{3/2}y).
\end{equation}
\item[(iii)] 
 If $q^2y\to \infty$ and $q\equiv 2\bmod 4$, then $\widetilde\phi\in\{\pm\frac{\pi}{2}\}$
\begin{equation}\label{conv3}
\frac{1}{\varphi(q)} \sum_{\substack{p\in\ZZ_q^\times \\ \phi_{p,q}=\widetilde\phi}} f\bigg(\frac{p}{q},\bigg(\frac{p}{q}+\i y,0\bigg)\bigg)
= \frac{1}{2}\int_{\TT^2} f\bigg(\xi, \widetilde\gamma_0\bigg(x+\i \frac{1}{q^2y}, \widetilde\phi\bigg)\,\widetilde n_-\bigg(\frac12\bigg)\bigg)\, d\xi\,dx+O(q^{3/2}y).
\end{equation}
\end{enumerate}
\end{prop}

The following lemmas will be helpful in proving Proposition \ref{lemEqui}.

\begin{lem}\label{fltlem}
Let $p,q\in\NN$, $\gcd(p,q)=1$ and $y\in\RR_{>0}$.
\begin{enumerate}
\item[(i)] If $q\equiv 0\bmod 4$, then $(\frac{p}{q}+\i y,0)\in\Del (-\frac{\overline p}{q}+\i \frac{1}{q^2 y}, \phi_{p,q})$ with $\e^{\i\phi_{p,q}/2}=\epsilon_p(\frac{q}{p})\e^{-\i\pi/4}$. 
\item[(ii)] If $q\equiv 1\bmod 2$, then $(\frac{p}{q}+\i y,0)\in\Del \, \widetilde\gamma_0\, (-\frac{\overline{4p}}{q}+\i \frac{1}{4 q^2 y}, \phi_{p,q})$, where $\overline{4p}$ denotes the inverse mod $q$, and $\e^{\i\phi_{p,q}/2}=\epsilon_q^{-1} (\frac{p}{q})\e^{-\i\pi/4}$. 
\item[(iii)] If $q\equiv 2 \bmod 4$, then $(\frac{p}{q}+\i y,0)\in\Del \, \widetilde\gamma_0\, (-\frac{\overline{2p}}{q/2}+\i \frac{1}{q^2 y}, \phi_{p,q}) \widetilde n_-(\frac12)$, where $\overline{2p}$ denotes the inverse mod $q/2$, and $\e^{\i\phi_{p,q}/2}=\epsilon_{q/2}^{-1}(\frac{2p}{q/2})\e^{-\i\pi/4}$. 
\end{enumerate}
\end{lem}

\begin{proof}
(i) Assume $p\equiv 3\bmod 4$.
We set $c=q$, $d=-p$ and select $a,b\in\ZZ$ such that $ad-bc=1$; this is always possible since $\gcd(p,q)=1$. Since $c\equiv 0\bmod 4$ and $ad-bc=1$, the condition $d\equiv 1\bmod 4$ implies $a\equiv 1\bmod 4$. Hence, with these choices, 
\begin{equation}\label{tildegam}
\widetilde\gamma =\bigg[ \begin{pmatrix} a & b \\ c & d \end{pmatrix}, \beta_\gamma \bigg]\in \Del,
\end{equation}
where $\beta_\gamma$ is as in \eqref{betafunct}.
A short calculation shows that
$\widetilde\gamma (\frac{p}{q}+\i y,0)=(\frac{a}{q}+\i \frac{1}{q^2 y}, \beta_\gamma(\frac{p}{q}+\i y))$, where 
\begin{equation}\label{inst}
\e^{\i\beta_\gamma(\frac{p}{q}+\i y)/2}=\bigg(\frac{c}{d}\bigg)\,\e^{\i \arg(\i c y)/2}=\bigg(\frac{q}{-p}\bigg)\, \e^{\i \pi/4}=\bigg(\frac{q}{p}\bigg)\, \e^{\i \pi/4} ;
\end{equation}
recall that we only need to determine $\beta_\gamma \bmod 4\pi$ since $[1,4\pi]\in\Del$.
We observe that $ad-bc=1$ implies $a\equiv -\overline p\bmod q$, and finally that the set
$\Del (-\frac{\overline p}{q}+\i \frac{1}{q^2 y}, \phi_{p,q}))$
is well defined, since 
\begin{equation}\label{inv}
\bigg[ \begin{pmatrix} 1 & \ZZ \\ 0 & 1 \end{pmatrix},  0 \bigg] \subset\Del .
\end{equation}

Assume now $p\equiv 1\bmod 4$. We set $c=-q$, $d=p$, and proceed as above. We find $\widetilde\gamma (\frac{p}{q}+\i y,0)=(-\frac{a}{q}+\i \frac{1}{q^2 y}, \beta_\gamma(\frac{p}{q}+\i y))$, but instead of \eqref{inst} we have
\begin{equation}\label{inst11}
\e^{\i\beta_\gamma(\frac{p}{q}+\i y)/2}=\bigg(\frac{c}{d}\bigg) \e^{\i \arg(\i c y)/2}=\bigg(\frac{-q}{p}\bigg)\, \e^{-\i \pi/4}=(-1)^{(p-1)/2} \bigg(\frac{q}{p}\bigg)\,\e^{-\i \pi/4}=\bigg(\frac{q}{p}\bigg)\,\e^{-\i \pi/4} .
\end{equation}
Note that this time $a\equiv \overline p\bmod q$.

(ii) Let $p'$ be a positive integer such that $p'\equiv -\overline{4p}\bmod q$, $y'=\frac{1}{4q^2y}$, and $\phi_{p,q}$ as in the statement (ii) of the lemma. 

Assume $q\equiv 1\bmod 4$. 
For $\widetilde\gamma$ as in \eqref{tildegam} and $c=4p'$, $d=q$, we have
\begin{equation}\label{b4}
\widetilde\gamma\widetilde\gamma_0 \bigg(\frac{p'}{q}+\i y',\phi_{p,q}\bigg) = \bigg(\frac{b}{q}+\i \frac{1}{4q^2 y'}, \phi_{p,q}+\beta_\gamma\bigg(-\frac{1}{4(\frac{p'}{q}+\i y')}\bigg)+\arg\bigg(\frac{p'}{q}+\i y'\bigg) \bigg) ,
\end{equation}
where
\begin{equation}\label{inst1}
\begin{split}
& \exp\bigg\{\frac{\i}{2}\bigg[\beta_\gamma\bigg(-\frac{1}{4(\frac{p'}{q}+\i y')}\bigg)+\arg\bigg(\frac{p'}{q}+\i y'\bigg)\bigg]\bigg\}\\
&=\bigg(\frac{c}{d}\bigg) \exp\bigg\{\frac{\i}{2}\bigg[ \arg\bigg( -\frac{c}{4(\frac{p'}{q}+\i y')}+d\bigg)+\arg\bigg(\frac{p'}{q}+\i y'\bigg)\bigg]\bigg\}\\
&=\bigg(\frac{4p'}{q}\bigg) \exp\bigg\{\frac{\i}{2}\bigg[ \arg\bigg( \frac{\i q y'}{\frac{p'}{q}+\i y'}\bigg)+\arg\bigg(\frac{p'}{q}+\i y'\bigg)\bigg]\bigg\}\\
&=\bigg(\frac{4p'}{q}\bigg) \e^{\i\pi/4}.
\end{split}
\end{equation}
Because $ad-bc=1$, we have $-4p'b\equiv 1\bmod q$. This implies $(\frac{4p'}{q})=(\frac{-b}{q})=(-1)^{(q-1)/2}(\frac{b}{q})=(\frac{b}{q})$. In view of \eqref{inv} the above formulae hold for any integer $p$ such that $p\equiv b\bmod q$. Thus we have shown
\begin{equation}
\widetilde\gamma\widetilde\gamma_0 \bigg(\frac{p'}{q}+\i y',\phi_{p,q}\bigg) = \bigg(\frac{p}{q}+\i y, 0 \bigg) .
\end{equation}
To check that the right hand side of (ii) is well defined, note that
\begin{equation}
\bigg[\begin{pmatrix} 1 & 0 \\ 4 & 1 \end{pmatrix}, \arg(4\,\cdot\, +1) \bigg] \widetilde\gamma_0=\widetilde\gamma_0 \bigg[\begin{pmatrix} 1 & -1 \\ 0 & 1 \end{pmatrix} , 0 \bigg] 
\end{equation}
and so
\begin{equation}
\widetilde\gamma_0 \bigg[\begin{pmatrix} 1 & \ZZ \\ 0 & 1 \end{pmatrix} , 0\bigg] 
=
\bigg[\begin{pmatrix} 1 & 0 \\ 4 & 1 \end{pmatrix}, \arg(4\,\cdot\, +1) \bigg]^\ZZ \widetilde\gamma_0
\subset \Del \widetilde\gamma_0.
\end{equation}
We conclude that $\Del \, \widetilde\gamma_0\, (-\frac{\overline{4p}}{q}+\i \frac{1}{4 q^2 y}, \phi_{p,q})$ is well defined.

The case $q\equiv 3\bmod 4$ follows analogously by picking $c=-4p'$, $d=-q$. Instead of \eqref{b4} we have
\begin{equation}
\widetilde\gamma\widetilde\gamma_0 \bigg(\frac{p'}{q}+\i y',\phi_{p,q}\bigg) = \bigg(-\frac{b}{q}+\i \frac{1}{4q^2 y'}, \phi_{p,q}+\beta_\gamma\bigg(-\frac{1}{4(\frac{p'}{q}+\i y')}\bigg)+\arg\bigg(\frac{p'}{q}+\i y'\bigg) \bigg) ,
\end{equation}
and \eqref{inst1} becomes
\begin{equation}\label{inst2}
\begin{split}
& \exp\bigg\{\frac{\i}{2}\bigg[\beta_\gamma\bigg(-\frac{1}{4(\frac{p'}{q}+\i y')}\bigg)+\arg\bigg(\frac{p'}{q}+\i y'\bigg)\bigg]\bigg\}\\
&=\bigg(\frac{-4p'}{-q}\bigg) \exp\bigg\{\frac{\i}{2}\bigg[ \arg\bigg( -\frac{\i q y'}{\frac{p'}{q}+\i y'}\bigg)+\arg\bigg(\frac{p'}{q}+\i y'\bigg)\bigg]\bigg\}\\
&=- \bigg(\frac{-4p'}{q}\bigg)\, \e^{-\i\pi/4}\\
&=-(-1)^{(q-1)/2} \bigg(\frac{4p'}{q}\bigg) \e^{-\i\pi/4}\\
&=\bigg(\frac{4p'}{q}\bigg) \e^{-\i\pi/4}.
\end{split}
\end{equation}
Now $ad-bc=1$ implies $4p'b\equiv 1\bmod q$, and we conclude as above for any $p$ such that $p\equiv b\bmod q$.

We deduce claim (iii) from (ii): 
Define $q_0=q/2$ and $p_0=\frac14(2p-q)$. Clearly $q_0=1\bmod 2$, $\gcd(p_0,q_0)=1$ and $\frac{p}{q} = \frac{p_0}{q_0} + \frac12$. In view of (ii) we have
\begin{equation}
\bigg(\frac{p}{q}+\i y,0\bigg)=\bigg(\frac{p_0}{q_0}+\i y,0\bigg)\, \widetilde n_-\bigg(\frac12\bigg)\in\Del \, \widetilde\gamma_0\, \bigg(-\frac{\overline{4p_0}}{q_0}+\i \frac{1}{4q_0^2 y}, \phi_{p_0,q_0}\bigg) \;\widetilde n_-\bigg(\frac12\bigg) ,
\end{equation}
where $\e^{i\phi_{p_0,q_0}/2}=\epsilon_{q_0}^{-1} (\frac{p_0}{q_0})\e^{-\i\pi/4}$. We conclude by noting that $(\frac{p_0}{q_0})=(\frac{(2p-q)/4}{q/2})=(\frac{\overline 2}{q/2})^2(\frac{2p-q}{q/2})=(\frac{2p}{q/2})$.
\end{proof}
\begin{lem}\label{klstlem}
Let $m,n\in\ZZ$ with $(m,n)\neq (0,0)$, $p,q\in\NN$, $\gcd(p,q)=1$ and $\epsilon>0$.
\begin{enumerate}
\item[(i)] If $q\equiv 0\bmod 4$ and $\sigma\in\{\pm 1,\pm\i\}$, then
$$
\bigg| \sum_{\substack{p\in\ZZ_q^\times \\ \epsilon_p (\frac{q}{p})=\sigma}} e\bigg(\frac{mp+n\overline p}{q} \bigg)\bigg| \leq 7 \gcd(m,n,q)^{1/2} q^{1/2} \tau(q).
$$
\item[(ii)] If $q\equiv 1\bmod 2$ and $\sigma\in\{\pm 1\}$, then
$$
\bigg| \sum_{\substack{p\in\ZZ_q^\times \\ (\frac{p}{q})=\sigma}} e\bigg(\frac{mp+n\overline p}{q} \bigg)\bigg| \leq 2 \gcd(m,n,q)^{1/2}  q^{1/2} \tau(q).
$$
\end{enumerate}
\end{lem}

\begin{proof}
(i) We have
\begin{equation}
\sum_{\substack{p\in\ZZ_q^\times \\ \epsilon_p (\frac{q}{p})=\sigma}} e\bigg(\frac{mp+n\overline p}{q} \bigg)
=\sum_{k\in\ZZ_4}   \sigma^{-k}  \sum_{p\in\ZZ_q^\times} \epsilon_p^k \bigg(\frac{q}{p}\bigg)^k e\bigg(\frac{mp+n\overline p}{q} \bigg) .
\end{equation}
For $k=0$ the inner sum is the Kloosterman sum
\begin{equation}
K(m,n,q)=\sum_{p\in\ZZ_q^\times} e\bigg(\frac{mp+n\overline p}{q} \bigg),
\end{equation}
and for $k=1$ the twisted Kloosterman sum
\begin{equation}
\widetilde K(m,n,q)=\sum_{p\in\ZZ_q^\times} \epsilon_p \bigg(\frac{q}{p}\bigg) e\bigg(\frac{mp+n\overline p}{q} \bigg).
\end{equation}
The case $k=-1$ yields the complex conjugate of $\widetilde K(-m,-n,q)$. The case $k=2$ yields
\begin{equation}
\sum_{p\in\ZZ_q^\times} \epsilon_p^2\, e\bigg(\frac{mp+n\overline p}{q} \bigg)
=\sum_{p\in\ZZ_q^\times}  e\bigg(\frac{p-1}{4}+\frac{mp+n\overline p}{q} \bigg)
=-\i\,K\bigg(m+\frac{q}{4},n,q\bigg).
\end{equation}

The claim follows now from the triangle inequality and the classical Weil bound $|K(m,n,q)|\leq \gcd(m,n,q)^{1/2}  q^{1/2} \tau(q)$ \cite{Esterman}, where $\tau(q)$ is the number of positive divisors of $q$, its analogue in the twisted case, $|\widetilde K(m,n,q)|\leq \gcd(m,n,q)^{1/2}  q^{1/2} \tau(q)$ see \cite{Chinen,Duke}, and the inequality $\gcd(m+\frac{q}{4},n,q)\leq  \gcd(4m+q,n,q)=\gcd(4m,n,q)\leq 4\gcd(m,n,q)$.

(ii) We have for $\sigma=\pm 1$,
\begin{equation}
\sum_{\substack{p\in\ZZ_q^\times \\ (\frac{p}{q})=\sigma}} e\bigg(\frac{mp+n\overline p}{q} \bigg)
=
\sum_{k\in\ZZ_2}   \sigma^{-k}  \sum_{p\in\ZZ_q^\times}  \bigg(\frac{p}{q}\bigg)^k e\bigg(\frac{mp+n\overline p}{q} \bigg) .
\end{equation}
The claim follows again from Weil's bound for the classical Kloosterman sum and the analogous bound $|S(m,n,q)|\leq \gcd(m,n,q)^{1/2}  q^{1/2} \tau(q)$ for the Sali\'e sum
\begin{equation}
S(m,n,q)=\sum_{p\in\ZZ_q^\times}  \bigg(\frac{p}{q}\bigg) e\bigg(\frac{mp+n\overline p}{q} \bigg),
\end{equation}
which is proved in \cite{Chinen}.
\end{proof}

\begin{proof}[The proof of Proposition \ref{lemEqui}]

We set
\begin{equation}
f_0(\xi,g):=
\begin{cases}
f(\xi, g\, \widetilde k(\widetilde\phi)) & \text{ in case (i),} \\
f(\xi, \widetilde\gamma_0\, g\, \widetilde k(\widetilde\phi)) & \text{ in case (ii),} \\
f(\xi, \widetilde\gamma_0\, g\, \widetilde k(\widetilde\phi)\, \widetilde n_-(\frac12)) & \text{ in case (iii),} 
\end{cases}
\end{equation}
where $\widetilde k(\widetilde\phi)=(\i, \widetilde\phi)$.
Because $\widetilde\gamma_0 \Del= \Del \widetilde\gamma_0$, and furthermore since right multiplication by a fixed element in $\USLR$ preserves $C_0^\infty$, we have $f_0\in\C_0^\infty\big(\TT\times\scrM\big)$.

In view of \eqref{inv} the function $f_0(\xi,(x+\i y,0))$ is periodic of period one not only in $\xi$ but also in $x$. Its Fourier expansion reads
\begin{equation}
f_0(\xi,(x+\i y,0))= \sum_{m,n\in\ZZ} a(m,n,y) \,e(m\xi+nx) ,
\end{equation}
where
\begin{equation}
a(m,n,y) = \int_{\TT^2} f_0(\xi,(x+\i y,0)) \,e(-m\xi-nx) \, d\xi\,dx.
\end{equation}
An argument similar to the one leading to Eq.~(6.17) in \cite{Marklof03} yields, for every $A,B\geq 0$,
\begin{equation}\label{fourierbound}
\big| a(m,n,y) \big| \leq C_{A,B} (1+|m|)^{-B} (1+|n|y)^{-A},
\end{equation}
where $C_{A,B}>0$ is a constant independent of $m,n$ and $y$.

We now consider the distinct cases:

(i) By Lemma \ref{fltlem} (i), we have
\begin{equation}
\begin{split}
\sum_{\substack{p\in\ZZ_q^\times \\ \phi_{p,q}=\widetilde\phi}} f\bigg(\frac{p}{q},\bigg(\frac{p}{q}+\i y,0\bigg)\bigg)  & =
\sum_{\substack{p\in\ZZ_q^\times \\ \phi_{p,q}=\widetilde\phi}} f_0\bigg(\frac{p}{q},\bigg(-\frac{\overline p}{q}+\i \frac{1}{q^2y}, 0 \bigg)\bigg) \\
& =\sum_{m,n\in\ZZ} a\bigg(m,n,\frac{1}{q^2y}\bigg) \sum_{\substack{p\in\ZZ_q^\times \\ \phi_{p,q}=\widetilde\phi}} e\bigg(\frac{mp-n\overline p}{q}\bigg) .
\end{split}
\end{equation}
In order to get the rate of convergence we need to bound
\begin{equation}\label{convergence}\sum_{(m,n)\neq(0,0)}(1+|m|)^{-100}\,(1+|n|\frac{1}{q^2y})^{-100}\,\frac{1}{\varphi(q)}\,\bigg|\sum_{\substack {p\in\ZZ_q^{\times}\\ \phi_{p,q}=\widetilde\phi}}\e\bigg(\frac{mp-n\bar p}{q}\bigg)\bigg|.
\end{equation}

 Note that the inner sum in Equation \eqref{convergence} is a Kloosterman sum and Lemma \ref{klstlem} (i) is can be employed. We now deal with different cases:

(a) The sum \eqref{convergence} restricted to $m\nequiv 0\bmod q$ is bounded above by
\begin{equation}\sum_{m,n}(1+|m|)^{-100}\;\bigg(1+|n|\frac{1}{q^2y}\bigg)^{-100}\,\frac{1}{\varphi(q)} \gcd(m, n, q)^{1/2}\tau(q)q^{1/2}.
\end{equation}
Since $\gcd(m,n,q)^{1/2}\leq |m|^{1/2}$, and $\sum_m (1+|m|)^{-100}|m|^{1/2}$ converges, and by combining the standard estimates $\varphi(q)=(q'-1)\,\varphi(a)$ and $\tau(q)=O(1)$ (since $q=a q'$ and $\gcd(q',a)=1$) for $q'$ sufficiently large, we then have 
\begin{equation}\begin{split}
&\ll\sum_{n}\bigg(1+|n|\frac{1}{q'^2y}\bigg)^{-100}\,\frac{\sqrt{q'}}{q'-1}\\
 &\ll q^{3/2}y,
\end{split}\end{equation}
where the implied constant depends on $a$.

(b) If the sum \eqref{convergence} is restricted to $m\equiv0\bmod q$, set $m=k q$ and $q=aq'$, we get the following bound:
\begin{equation}\begin{split}
&\sum_{k}(1+|kq|)^{-100}\,\sum_{n}\bigg(1+|n|\frac{1}{q^2y}\bigg)^{-100}\frac{1}{\varphi(q)}\gcd(kq,n,q)^{1/2}\,\tau(q)\,q^{1/2}\\
&\ll\frac{1}{q'-1}\bigg[\sum_{\substack{n\\q'|n}}\bigg(1+|n|\frac{1}{q'^2y}\bigg)^{-100}\,q'+\sum_{\substack{n\\q'\nmid n}}\bigg(1+|n|\frac{1}{q'^2y}\bigg)^{-100}\,q'^{1/2}\bigg].\\
\end{split}\end{equation}
The sums above are Riemann sums, then we have the bound
\begin{equation}\begin{split}
&\frac{1}{q'-1}\bigg[\sum_{t}\bigg(1+|t|\frac{1}{q' y}\bigg)^{-100}\,q'+\sum_{\substack{n\\q'\nmid n\\|n|\leq q'^2 y}}q'^{1/2}\bigg]\ll q^{3/2}y.
\end{split}\end{equation}
The implied constant again depends on $a$.

(ii) By Lemma \ref{fltlem} (ii), we have
\begin{equation}
\begin{split}
\sum_{\substack{p\in\ZZ_q^\times \\ \phi_{p,q}=\widetilde\phi}} f\bigg(\frac{p}{q},\bigg(\frac{p}{q}+\i y,0\bigg)\bigg) & =
\sum_{\substack{p\in\ZZ_q^\times \\ \phi_{p,q}=\widetilde\phi}} f_0\bigg(\frac{p}{q},\bigg(-\frac{\overline{4p}}{q}+\i \frac{1}{4q^2 y}, 0 \bigg)\bigg) \\
& =\sum_{m,n\in\ZZ} a\bigg(m,n,\frac{1}{4q^2 y}\bigg) \sum_{\substack{p\in\ZZ_q^\times \\ \phi_{p,q}=\widetilde\phi}} e\bigg(\frac{mp-\overline4 n\,\overline{p}}{q}\bigg) .
\end{split}
\end{equation}
We conclude as in case (i).

(iii) By Lemma \ref{fltlem} (iii), we have
\begin{equation}
\begin{split}
\sum_{\substack{p\in\ZZ_q^\times \\ \phi_{p,q}=\widetilde\phi}} f\bigg(\frac{p}{q},\bigg(\frac{p}{q}+\i y,0\bigg)\bigg) & =
\sum_{\substack{p\in\ZZ_q^\times \\ \phi_{p,q}=\widetilde\phi}} f_0\bigg(\frac{p}{q},\bigg(-\frac{\overline{2p}}{q/2}+\i \frac{1}{q^2y}, 0 \bigg)\bigg) \\
& =\sum_{m,n\in\ZZ} a\bigg(m,n,\frac{1}{q^2y}\bigg) \sum_{\substack{p\in\ZZ_q^\times \\ \phi_{p,q}=\widetilde\phi}} e\bigg(\frac{mp-\overline2 n\,\overline p}{q/2}\bigg) .
\end{split}
\end{equation}
We proceed as above, but use the bound in Lemma \ref{klstlem} (ii) with $q/2$ in place of $q$.
\end{proof}

\begin{prop}\label{thmEqui}
Let $f\in\C^\infty_{0}({\TT\times\scrM})$.
Then, for $q^2y\rightarrow \infty$, one has

\begin{enumerate}
\item[(i)] If $q\equiv0\bmod4$, then
\begin{equation}\begin{split} & \int_{\TT^2} f\bigg(\xi, \bigg(x+\i \frac{1}{q^2y}, 0\bigg)\bigg)\, d\xi\,dx =
\frac{1}{8\pi^2}
\int_{\TT\times{\scrM}} f(\xi,z,\widetilde\phi)\;d\xi\;d\mu(z,\widetilde\phi)\\
&+
\frac{1}{8 \pi^2}
\bigg[\sum_{\nu=0}^\infty h_{2\nu}(0) \int_{\TT\times{\scrM}} f(\xi,z,\widetilde\phi) 
\Re\{ \theta_{2\nu}(z,\widetilde\phi)\}\,
d\xi\; d\mu(z,\widetilde\phi)\bigg]\; \bigg(\frac{1}{q^2y}\bigg)^{1/4} +o\bigg(\frac{1}{\sqrt{q^2y}}\bigg).
\end{split}\end{equation}
\item[(ii)] If $q\equiv1\bmod2$, then
\begin{equation}\begin{split}
& \int_{\TT^2} f\bigg(\xi, \widetilde\gamma_0\bigg(x+\i \frac{1}{4q^2y}, 0\bigg)\bigg)\, d\xi\,dx =
\frac{1}{8\pi^2}
\int_{\TT\times{\scrM}} f(\xi,z,\widetilde\phi)\;d\xi\;d\mu(z,\widetilde\phi)\\
&+
\frac{1}{8 \pi^2}
\bigg[\sum_{\nu=0}^\infty h_{2\nu}(0) \int_{\TT\times{\scrM}} f(\xi,z,\widetilde\phi) 
\Re\{ \theta_{2\nu}(z,\widetilde\phi)\}\,
d\xi\; d\mu(z,\widetilde\phi)\bigg]\; \bigg(\frac{1}{4q^2y}\bigg)^{1/4} +o\bigg(\frac{1}{\sqrt{q^2y}}\bigg).
\end{split}
\end{equation}
\item[(iii)] If $q\equiv2\bmod4$, then
\begin{equation}\begin{split}
& \int_{\TT^2} f\bigg(\xi, \widetilde\gamma_0\bigg(x+\i \frac{1}{q^2y}, 0\bigg)\widetilde n_-\bigg(\frac12\bigg)\bigg)\, d\xi\,dx =
\frac{1}{8\pi^2}
\int_{\TT\times{\scrM}} f(\xi,z,\widetilde\phi)\;d\xi\;d\mu(z,\widetilde\phi)\\
&+
\frac{1}{8 \pi^2}
\bigg[\sum_{\nu=0}^\infty h_{2\nu}(0) \int_{\TT\times{\scrM}} f(\xi,z,\widetilde\phi) 
\Re\{ \theta_{2\nu}(z,\widetilde\phi)\}\,
d\xi\; d\mu(z,\widetilde\phi)\bigg]\; \bigg(\frac{1}{q^2y}\bigg)^{1/4} +o\bigg(\frac{1}{\sqrt{q^2y}}\bigg).
\end{split}\end{equation}
\end{enumerate}
\end{prop}

\begin{proof}The proof of case (i) is a direct consequence of Theorem 5.1 for closed horocycles in the metaplectic group \cite{Marklof} (which generalizes Sarnak's work \cite{Sarnak81} ), corresponding to the cusp at $\infty$. For the remaining cases, the statement of theorem holds for the test function 
\begin{equation}
f_1(\xi,z,\widetilde\phi):=
\begin{cases}
f(\xi, \widetilde\gamma_0\,(z,\widetilde\phi)) & \text{ in case (ii),} \\
f(\xi, \widetilde\gamma_0(z,\widetilde\phi)\, \widetilde n_-(\frac12)) & \text{ in case (iii),} 
\end{cases}
\end{equation}
 since $\widetilde\gamma_0 [g, \beta_g]\widetilde\gamma_0^{-1}\in\Del$ when $[g,\beta_g]$ in $\Del$ and right multiplication a fixed element in $\USLR$ preserves $C_0^\infty$, we have $f_1\in\C_0^\infty\big(\TT\times\scrM\big)$. Hence the cusps with respect to at 0 and 1/2 are represented as the standard cusp at $\infty$.
\end{proof}

The proof of Theorem \ref{mtplctcEq} is now easily completed by combining Proposition \ref{lemEqui} and Proposition \ref{thmEqui}.

\section{Main theorem}\label{maintheorem}
We generalize Theorem \ref{thm1} by considering the joint distribution of the classical Gauss sum $g_1(p,q)$ and the short incomplete Gauss sum $g_f(p,q,N)$. The classical Gauss sum is
\begin{equation}
g_1(p,q)=\sum_{h\bmod q} e_q(p h^2),
\end{equation}
where $p<q$ are coprime integers, and $e_q(x)=\e^{2\pi\i x/q}$.
The classical Gauss sum can be evaluated explicitly:
\begin{equation}\label{GS} 
g_1(p,q)=
\begin{cases}
(1+\i)\; \epsilon_p^{-1} (\frac{q}{p})\; \sqrt q
& \text{if $q\equiv 0\bmod 4$} \\ 
\epsilon_q (\frac{p}{q}) \;\sqrt q
& \text{if $q\equiv 1\bmod 2$}\\
0 & \text{if $q\equiv 2\bmod 4$}
\end{cases}
\end{equation}
where $(\frac{a}{b})$ and $\epsilon_a$ are defined as in Section \ref{preliminaries}.

Let us now consider the theta function
\begin{equation}
\Theta_f(z,\phi)=y^{1/4}\sum_{h\in\ZZ} f_{\phi}(h y^{1/2})\,e(x h^2)
\end{equation}
where $$f_\phi(t)=\sum_{nu=0}^\infty \hat f(\nu)\,\e^{-\i(2\nu+1)\phi/2}\,h_\nu(t)$$ with $\nu$-th Hermite coefficient $\hat f=(f,h_{\nu})$ (cf. \eqref{hermite}). Marklof \cite{Marklof99} has showed that $\Theta_f(z,\phi)$ is a function on $\scrM$.
For the rational points $\frac pq+\i\frac{1}{N^2}$ on the metaplectic horocycles we clearly have
\begin{equation}\label{ThetaGauss1}
\Theta_f\bigg(\frac pq+\i \frac{1}{N^2},0\bigg)=\frac{g_f(p,q,N)}{\sqrt N}
\end{equation}
with $y=N^{-2}$.
The limit distribution $\nu_f$ in Theorem \ref{thm1} is given by
\begin{equation}\int_{\CC} F(z)\,\nu_f(d z)=\int_{\scrM} F(\Theta(g))\,d\mu(g).
\end{equation}
That is, the limit distribution of short incomplete Gauss sums is characterized by the theta series $\Theta_f(g)$ with $g=(z,\phi)$ uniformly distributed on $\scrM$ with respect to Haar measure $d \mu(g)=\frac{dx\,dy\,d\phi}{y^2}$. We define furthermore the following random variables: The random variable $X$ takes the values $\pm1 \pm\i$ with equal probability and the random variable $Y$ takes the values $\pm 1$ with equal probability.

The symbol $\xrightarrow{d}$ denotes convergence in distribution.

\begin{thm}\label{shortgauss}
Fix $a\in\NN$ and a subset $\scrD\subset\TT$ with $|\partial\scrD|=0$, and $f:\RR\to\RR$ Riemann integrable and compactly supported. For each $q\in\NN$, choose $p\in\ZZ_q^\times\cap q\scrD$ at random with uniform probability. As $q\to \infty$ along any subsequence with $q=aq'$, $q'$ prime, $\frac{N}{q}\to0$ and $\frac{N^{4/3}}{q}\to\infty$, then:

\begin{enumerate}
\item[(i)] If $q\equiv 0\bmod 4$, then
\begin{equation}
\bigg( \frac{g_1(p,q)}{\sqrt{q}}, \frac{g_f(p,q,N)}{\sqrt{N}} \bigg) \xrightarrow{d} (X,\Theta_f) .
\end{equation}
\item[(ii)] 
If $q\equiv 1\bmod 2$, then
\begin{equation}
\bigg( \frac{g_1(p,q)}{\epsilon_q\sqrt q}, \frac{g_f(p,q,N)}{\sqrt N} \bigg) \xrightarrow{d} (Y,\Theta_f) .
\end{equation}
\item[(iii)] 
If $q\equiv 2\bmod 4$, then
\begin{equation}
\bigg( \frac{g_1(2p,q/2)}{\epsilon_{q/2}\sqrt{q/2}}, \frac{g_f(p,q,N)}{\sqrt N} \bigg) \xrightarrow{d} (Y, \Theta_f).
\end{equation}
\end{enumerate}
\end{thm}

\section{Proof of Theorem 4 for smooth weights}\label{smooth}

We now give the proof of Theorem \ref{shortgauss} for that the weight function $f$ of $g_f(p,q,N)$ is of Schwartz class, which means $f$ is smooth and all its derivatives are rapidly decreasing.

\begin{proof}
{\bf Case i:} $q\equiv 0 \bmod 4$. This case can be formulated as follows: For any bounded continuous $F:\CC\to\RR$ 

\begin{equation}\frac{1}{\#(\ZZ^\times_q\cap q\scrD)}\sum_{\substack {p\in(\ZZ^\times_q\cap q\scrD) \\ g_1(p,q)=\sqrt{q} \sigma}}F\bigg(\frac{g_f(p,q,N)}{\sqrt{N}}\bigg)\rightarrow\frac{1}{32\pi^2}\int_{\scrM}F(\Theta_f(g))\,d \mu(g)
\end{equation}
where $\sigma \in \{\pm1 \pm \i\}$.
This then follows from
\begin{equation}\label{say}
\frac{1}{\varphi(q)}\sum_{\substack{p\in\ZZ_q^\times \\ \phi{p,q}=\widetilde\phi}}
\chi_\scrD\bigg(\frac{p}{q}\bigg) F\bigg( \frac{g_f(p,q,N)}{\sqrt N} \bigg)\rightarrow\frac{|\scrD|}{32\pi^2}\int_{\scrM}F(\Theta_f(g))\,d\mu(g) .
\end{equation}

By \eqref{ThetaGauss1}, this corresponds to 
\begin{equation}
\frac{1}{\varphi(q)}\sum_{\substack{p\in\ZZ_q^\times \\ \phi_{p,q}=\widetilde\phi}}
\chi_\scrD\bigg(\frac{p}{q}\bigg) F\bigg(\Theta\bigg(\frac{p}{q}+\i \frac{1}{N^2}, 0\bigg)\bigg) \rightarrow\frac{|\scrD|}{32\pi^2}\int_{\scrM}F(\Theta(g))\,d \mu(g) .
\end{equation}
We now choose the test function
\begin{equation}
f(\xi,z,\widetilde\phi) = \chi_\scrD(\xi)\,F(\Theta_f(z,\widetilde\phi)),
\end{equation}
with $g=(z,\widetilde\phi)$. The proof then follows by Theorem \ref{mtplctcEq} (i), and a standard approximation argument in which we approximate $\chi_{\scrD}$ by continuous functions.

{\bf Case ii:} $q\equiv 1 \bmod 2$. In this case the statement reduces to 

\begin{equation}\frac{1}{\#(\ZZ^\times_q\cap q\scrD)}\sum_{\substack {p\in(\ZZ^\times_q\cap q\scrD) \\ g_1(p,q)=\pm\sqrt{q}\epsilon_q}}\;F\bigg(\frac{g_f(p,q,N)}{\sqrt{N}}\bigg)\rightarrow\frac{1}{16\pi^2}\int_{\scrM}F(\Theta_f(g))d\;\mu(g),
\end{equation}
that is
\begin{equation}\label{eq2}
\frac{1}{\varphi(q)}\sum_{\substack{p\in\ZZ_q^\times \\ \phi{p,q}=\widetilde\phi}}
\chi_\scrD\bigg(\frac{p}{q}\bigg)\,F\bigg( \frac{g_f(p,q,N)}{\sqrt N} \bigg)\rightarrow\frac{|\scrD|}{16\pi^2}\int_{\scrM}F(\Theta_f(g))\,d\mu(g) .
\end{equation}
The rest is analogous to case (i) except that we use Theorem \ref{mtplctcEq} (ii).

{\bf Case iii:} $q\equiv 2 \bmod 4$. We now deduce the statement of theorem to 

\begin{equation}\frac{1}{\#(\ZZ^\times_q\cap q\scrD)}\sum_{\substack {p\in(\ZZ^\times_q\cap q\scrD) \\ g_1(2p,q/2)=\pm\sqrt{q/2}\epsilon_{q/2}}}F\bigg(\frac{g_f(p,q,N)}{\sqrt{N}}\bigg)\rightarrow\frac{1}{16\pi^2}\int_{\scrM}F(\Theta_f(g))\,d \mu(g).
\end{equation}

In view of substituting $q=2 q_0$ and $p=2 p_0+ q_0$, we instead of \eqref{eq2} have 
\begin{equation}
\frac{1}{\varphi(q)}\sum_{\substack{p\in\ZZ_{q_0}^\times \\ \phi{p_0,q_0}=\widetilde\phi}}
\chi_\scrD\bigg(\frac{p_0}{q_0}+\frac12\bigg)\,F\bigg( \frac{g_f(2 p_0+q_0,2 q_0,N)}{\sqrt N} \bigg)\rightarrow\frac{|\scrD|}{16\pi^2}\int_{\scrM}F(\Theta_f(g))\,d\mu(g) .
\end{equation}

This then reads

\begin{equation}
\frac{1}{\varphi(q)}\sum_{\substack{p\in\ZZ_{q_0}^\times \\ \theta{p_0,q_0}=\widetilde\theta}}
\chi_\scrD\bigg(\frac{p_0}{q_0}+\frac12\bigg)\, F\bigg( \Theta_f\bigg(\bigg(\frac{p_0}{q_0}+\frac12\bigg)+\i\frac{1}{N^2}, 0 \bigg)\bigg)\rightarrow\frac{|\scrD|}{16\pi^2}\int_{\scrM}F(\Theta_f(g))\,d\mu(g) .
\end{equation}

We employ Theorem \ref{mtplctcEq} (iii), it hence follows with the same strategy as in case (ii).
\end{proof}

\section{Proof of Theorem 4 for Riemann integrable weights}\label{riemann}
The estimate of the mean square
\begin{equation}
M_{f}(q,N)=\frac{1}{\varphi(q)\,|\scrD|} \sum_{p\in\ZZ_q^\times\cap q\scrD} |g_f(p,q,N)|^2
\end{equation}
is the key ingredient for the proof of Theorem \ref{shortgauss} for Riemann integrable functions $f$.

\begin{lem}\label{RIlem}
Fix $a\in\NN$. Then there exists a constant $C_a>0$ such that as $q\to \infty$ along any subsequence with $q=aq'$, $q'$ prime,
$\frac{N}{q}\to 0$, and for every compactly supported Riemann integrable function $f:\RR\to\RR$, we have
\begin{equation}\label{varia}
\limsup_{q\to\infty} \frac{M_{f}(q,N)}{N} \leq \frac{C_a}{|\scrD|}\, \| f \|_2^2.
\end{equation}
\end{lem} 

\begin{proof}
We have
\begin{equation}\label{mean}
\begin{split}
\sum_{p\in\ZZ_q^\times\cap q \scrD} |g_f(p,q,N)|^2 & \leq \sum_{m\in\ZZ_q} |g_f(m,q,N)|^2 \\
& =\sum_{p\in\ZZ_q}\bigg|\sum_{h\in\ZZ}f\bigg(\frac{h}{N}\bigg)\,e_q(p h^2)\sum_{h'\in\ZZ}\overline{f\bigg(\frac{h'}{N}\bigg)\,e_q(p h'^2)}\bigg|\\
& \leq q \sum_{\substack{h,h'\in\ZZ\\ h^2\equiv {h'}^2\bmod q}} \bigg|f\bigg(\frac{h}{N}\bigg)\overline{ f\bigg(\frac{h'}{N}\bigg)} \bigg| .
\end{split}
\end{equation}

Since $q=a q'$ with $\gcd(a,q')=1$, by the Chinese remainder theorem the sum on the right hand side of equation \eqref{mean} is therefore
\begin{equation}
\sum_{\substack{h,h'\in \ZZ\\ h^2\equiv h'^2\bmod q'\\h^2\equiv h'^2\bmod a}}\bigg|f\bigg(\frac{h}{N}\bigg)\,\overline{ f\bigg(\frac{h'}{N}\bigg)} \bigg|.
\end{equation}
We then obtain 
\begin{equation}\begin{split}
M_{f}(q,N)
&\leq\frac{q}{|\scrD|\,\varphi(q)}\sum_{\substack{h,h'\in \ZZ\\ h^2\equiv h'^2\bmod q'}}\bigg|f\bigg(\frac{h}{N}\bigg) \,\overline{f\bigg(\frac{h'}{N}}\bigg) \bigg|\\
&\leq\frac{q}{|\scrD|\,\varphi(q)}\sum_{\substack{h,h'\in \ZZ \\ h\equiv h'\bmod q'}}\bigg(\bigg|f\bigg(\frac{h}{N}\bigg) \,\overline{f\bigg(\frac{h'}{N}\bigg)} \bigg|+\bigg|f\bigg(\frac{h}{N}\bigg)\,\overline{f\bigg(-\frac{h'}{N}\bigg)} \bigg|\bigg)\\
&=\frac{q}{|\scrD|\,\varphi(q)}\sum_{h}\sum_{k}\bigg(\bigg|f\bigg(\frac{h}{N}\bigg)\,\overline{f\bigg(\frac{h+kq'}{N}\bigg)} \bigg|+\bigg|f\bigg(\frac{h}{N}\bigg)\,\overline{f\bigg(-\frac{h+kq'}{N}\bigg)} \bigg|\bigg).\\
\end{split}\end{equation}

Note that for $k \neq 0$ and $N$ large enough, the supports of $f$ and $f$ shifted by $k\,q'/N$ do not overlap, and therefore the contributions from those terms disappear. Hence, only the terms with $k=0$ contribute so that
 \begin{equation}\begin{split}
 \lim_{q\to\infty}\sup \frac{M_{f}(q,N)}{N}
&\leq\lim_{N\to\infty}\sup\frac{\widetilde C_a}{|\scrD|\,N}\,\sum_{h}\bigg(\bigg|f\bigg(\frac{h}{N}\bigg) \bigg|^2+\bigg|f\bigg(\frac{h}{N}\bigg) \,\overline{f\bigg(-\frac{h}{N}\bigg)} \bigg|\bigg)\\
&\leq \frac{2\,\widetilde C_a}{|\scrD|}\,\left\|f\right\|_2^2,\\
\end{split}\end{equation}
by the Cauchy-Schwartz inequality.
\end{proof}
Lemma below states the relatively compactness by showing that the sequence probability measure given by short incomplete Gauss sums is tight. This furthermore implies that every sequence has a convergent subsequence. 
\begin{lem}\label{lemB}
 Let $f:\RR\to\RR$ be Riemann integrable with compact support. Then, for every $\epsilon>0$, $\delta>0$ there exists a smooth function $\widetilde f$ such that for the subsequence of $q$ specified in Lemma \ref{RIlem},
\begin{equation}
\limsup_{q\to\infty } \frac{1}{\varphi(q)} \big|\{ p\in\ZZ_q^\times : N^{-1/2} |g_f(p,q,N)-g_{\widetilde f}(p,q,N)|> \delta \}\big| < \epsilon .
\end{equation}
\end{lem}
\begin{proof} 
By Chebyshev's inequality we have
\begin{equation}\label{chebyshev}
\limsup_{q\to\infty} \frac{1}{\varphi(q)} \big|\{ p\in\ZZ_q^\times : N^{-1/2} |g_f(p,q,N)|> \delta \}\big| < \frac{M_f(q,N)}{\delta^2\;N}.
\end{equation}

By Lemma \ref{RIlem}, there exists $K_\epsilon>0$ such that
\begin{equation}\label{chebyshev2}
\limsup_{q\to\infty } \frac{1}{\varphi(q)} \big|\{ p\in\ZZ_q^\times : N^{-1/2} |g_f(p,q,N)|> K_\epsilon \}\big| < \epsilon\; \|f\|_2^2.
\end{equation}

Since $g_f(p,q,N)-g_{\widetilde f}(p,q,N)=g_{f-\widetilde f}(p,q,N)$ and $f-\widetilde f$ is Riemann integrable, we get
\begin{equation}
\limsup_{q\to\infty} \frac{1}{\varphi(q)} \big|\{ p\in\ZZ_q^\times : N^{-1/2} |g_f(p,q,N)-g_{\widetilde f}(p,q,N)|> \delta \}\big| < 
\frac{ M_{f-\tilde f}(q,N)}{\delta^2 \;N}.
\end{equation}
The proof then follows by \eqref{chebyshev} and\eqref{chebyshev2}.
\end{proof}

\begin{proof}[The proof of Theorem \ref{shortgauss}]

We only go through the case $q\equiv0\bmod 4$, the other cases are analogous.

The tightness argument given by Lemma \ref{lemB} tells us that any sequence of $q\to\infty$ with $q=a q'$ contains a subsequence $\{q_j\}$ with the property: there is a probability measure $\nu$ (depending on the sequence chosen, $f$ and $\scrD$) on $\{\pm\frac{\pi}{2},\pm\frac{3 \pi}{2}\}\times\CC$ such that for any $\widetilde\phi\in\{\pm\frac{\pi}{2},\pm\frac{3\pi}{2}\}$ and any bounded continuous function $F:\CC\to\RR$ we have
\begin{equation}\label{lili}
\lim_{j\to\infty}\frac{1}{|\scrD|\,\varphi(q_j)}\sum_{\substack{p\in\ZZ_{q_j}^\times\cap q_j\scrD \\ \phi_{p,q_j}=\widetilde\phi}} F\bigg( \frac{g_f(p,q_j,N)}{\sqrt N} \bigg) = 
\int_{\CC}F(z)\, \nu_f(\widetilde\phi,dz) .
\end{equation}

We claim that for every $F\in\C_0^\infty(\CC)$
\begin{equation}\label{IF}
\lim_{q\to\infty}\frac{1}{|\scrD|\,\varphi(q)}\sum_{\substack{p\in\ZZ_{q}^\times\cap q\scrD \\ \phi_{p,q}=\widetilde\phi}} F\bigg( \frac{g_f(p,q,N)}{\sqrt N} \bigg)= \int_{\CC}F(z)\,\nu_f(\widetilde\phi,dz) 
\end{equation}
holds and it thus implies that $\nu$ is unique and the full sequence of $q$ converges. 

To prove the existence of limit \eqref{IF}, notice that since $F\in\C_0^\infty(\CC)$ we have $|F(w)-F(z)| \leq C \min\{1,|w-z|\}$ for some constant $C>0$. Therefore, for $\widetilde f$, $\delta$, $\epsilon$ as in Lemma \ref{lemB}, we have
\begin{equation}
\begin{split}\label{ineq000}
& \frac{1}{|\scrD|\,\varphi(q)}\sum_{\substack{p\in\ZZ_{q}^\times\cap q\scrD \\ \phi_{p,q}=\widetilde\phi}} \bigg| F\bigg( \frac{g_f(p,q,N)}{\sqrt N} \bigg)-
F\bigg( \frac{g_{\widetilde f}(p,q,N)}{\sqrt N} \bigg) \bigg| \\
& \leq \frac{C}{|\scrD|\,\varphi(q)}\sum_{\substack{p\in\ZZ_{q}^\times\cap q\scrD \\ \phi_{p,q}=\widetilde\phi }} \min\bigg\{ 1, \bigg| \frac{g_f(p,q,N)}{\sqrt N} -\frac{g_{\widetilde f}(p,q,N)}{\sqrt N} \bigg|\bigg\} \\
& \leq \frac{C}{|\scrD|\,\varphi(q)}\sum_{p\in\ZZ_{q}^\times} \min\bigg\{ 1, \bigg| \frac{g_f(p,q,N)}{\sqrt N} -\frac{g_{\widetilde f}(p,q,N)}{\sqrt N} \bigg|\bigg\} \\
& \leq \frac{C}{|\scrD|}\,(\delta +\epsilon ) .
\end{split}
\end{equation}
The sequence 
\begin{equation}\label{cauchy}
\frac{1}{|\scrD|\,\varphi(q)}\sum_{\substack{p\in\ZZ_{q}^\times\cap q\scrD \\ \phi_{p,q}=\widetilde\phi}} F\bigg( \frac{g_{\widetilde f}(p,q,N)}{\sqrt N} \bigg)
\end{equation}
defines a Cauchy sequence, as \eqref{IF} is satisfied for the smooth function $\widetilde f$ by Theorem \ref{shortgauss} for the smooth case. By the upper bound \eqref{ineq000}, the triangle inequality and the fact that \eqref{cauchy} is a Cauchy sequence, it is now observed that the sequence
\begin{equation}\label{say2}
\frac{1}{|\scrD|\,\varphi(q)}\sum_{\substack{p\in\ZZ_{q}^\times\cap q\scrD \\ \phi_{p,q}=\widetilde\phi}} F\bigg( \frac{g_f(p,q,N)}{\sqrt N} \bigg)
\end{equation}
is also a Cauchy sequence, therefore the claim is proved. We have thus shown that $\nu$ is unique and the full sequence of $q$ converges for every bounded continuous $F$.

Since $\widetilde f$ converges to $f$,
\begin{equation}\lim_{q\to\infty }\frac{1}{|\scrD|\,\varphi(q)}\sum_{\substack{p\in\ZZ_{q}^\times\cap q\scrD \\ \phi_{p,q}=\widetilde\phi}} F\bigg( \frac{g_{\widetilde f}(p,q,N)}{\sqrt N} \bigg) \to \lim_{q\to\infty}\frac{1}{|\scrD|\,\varphi(q)}\sum_{\substack{p\in\ZZ_{q}^\times\cap q\scrD \\ \phi_{p,q}=\widetilde\phi}} F\bigg( \frac{g_ f(p,q,N)}{\sqrt N} \bigg)
\end{equation}
holds by the bound \eqref{ineq000}. This concludes the proof of Theorem \ref{shortgauss} for Riemann integrable case.

\end{proof}

\section{Numerics}\label{numerics}
In Figures \ref{fig6007} and \ref{figboundry}, the value distribution of the real and imaginary parts of the short incomplete Gauss sum $\frac{g_f(p,q,N)}{\sqrt N}$, where $f$ is a characteristic function of the unit interval $[0,1]$, were given by histograms for $N=\left\lfloor \frac{1}{\sqrt{7}}\, q^{7/8}\right\rfloor$ and $N=\left\lfloor \frac{1}{\sqrt{7}}\, q^{3/4}\right\rfloor$, respectively. The curves show the probability density of the real and imaginary parts of $S_f(x,N)$ sampled at 400,000 random points $x$ in $[0,1]$ and truncated at $n=5000$.

In Figure \ref{figN=q}, we are dealing with the normalized short incomplete gauss sums for $N=\left\lfloor \frac{1}{\sqrt{7}}\,q\right\rfloor$. For this particular $N$, the value distribution of the short incomplete Gauss sums corresponds to the case of $N\asymp q$, see \cite{DemirciMarklof}.

The sum we were concerned with in \cite{DemirciMarklof} is

\begin{equation}
g_\varphi(p,q)=\sum_{h=0}^{q-1} \varphi\bigg(\frac{h}{q}\bigg)\, e_q(p h^2),
\end{equation}
where the weight function $\varphi$ is periodic with period one.
Then we have the following limiting distribution for the long incomplete Gauss sum for $q\equiv 1\bmod 2$
\begin{equation}
\bigg( \frac{g_1(p,q)}{\epsilon_q\sqrt q}, \frac{g_\varphi(p,q)}{g_1(p,q)} \bigg) \xrightarrow{d} (Y, G_\varphi)
\end{equation}
where 
$$G_\varphi(x)=\sum_{n\in\ZZ} \widehat \varphi_{n} \, e(n^2 x)$$
with Fourier coefficient $\widehat \varphi$ and $Y$ takes the values $\pm 1$ with equal probability.

From this point of view, we have the following relation

\begin{equation}\begin{split}
\frac{g_f(p,q,N)}{\sqrt{N}}&=\frac{1}{\sqrt{N}}\sum_{h=1}^{N}e_q(p h^2)\\
&=\frac{7^{1/4}}{\sqrt{q}}\sum_{h\in\ZZ} \varphi\bigg(\frac{h}{q}\bigg)\, e_q(p h^2)\\
&=7^{1/4}\,\frac{g_1(p,q)}{\epsilon_q\sqrt q}\,\frac{g_\varphi(p,q)}{g_1(p,q)}\,\epsilon_q.\\
\end{split}\end{equation}

For $q=6029\equiv 1\bmod4$ we have $\epsilon_q=1$. We then have the convergence $\frac{g_f(p,q,N)}{\sqrt N}\to 7^{1/4}\,Y \,G_\varphi $. Therefore the histograms in Figure \ref{figN=q} represent the value distribution of the real and imaginary parts the short incomplete Gauss sum $\frac{g_f(p,q,N)}{\sqrt N}$ for $N=\left\lfloor \frac{1}{\sqrt 7}\, q\right\rfloor$ and the curves indicate a numerical approximation to the real and imaginary parts of the probability density function of the random variable $7^{1/4}\,Y\, G_\varphi$.

\textsl{Acknowledgments.}I am grateful to my supervisor Jens Marklof for suggesting the problem and his support during the preparation of this paper. I would also like to thank Julia Brandes and Eugen Keil for their comment on the first draft of the paper.

\end{document}